\documentclass[11pt]{paper}
\usepackage{fullpage}
\usepackage{url}

\usepackage {amssymb}
\usepackage {amsmath}
\usepackage {amsthm}
\usepackage [usenames] {color}

\usepackage {enumerate}
\usepackage {plaatjes}
\usepackage {cite}
\usepackage{float}
\usepackage [left,pagewise] {lineno}
\usepackage {xspace}
\usepackage {wrapfig}

\definecolor {infocolor} {rgb} {0.6,0.6,0.6}

\definecolor {sepia} {rgb} {0.75,0.30,0.15}
\everymath{\color{sepia}}

\floatstyle{ruled}
\newfloat{algorithm}{thp}{alg}
\floatname{algorithm}{Algorithm}

\newcommand {\mathset} [1] {\ensuremath {\mathbb {#1}}}
\newcommand {\R} {\mathset {R}}

\newtheorem {theorem} {Theorem}[section]

\newtheorem {corollary}[theorem] {Corollary}

\newtheorem{definition}[theorem]{Definition}
\newtheorem{lemma}[theorem]{Lemma}

\title{\LARGE  Minimal Controllability Problems}
 \author{Alex Olshevsky\thanks{Alex Olshevsky is with the Department of Industrial and Enterprise Systems Engineering, University of Illinois at Urbana-Champaign, email: \texttt{aolshev2@illinois.edu}}  }
\begin{document}

\maketitle

\begin{abstract} Given a linear system, we consider the problem of finding a small set of variables to affect with an input so
that the resulting system is controllable. We show that this problem is NP-hard; indeed, we show that even approximating the minimum 
number of variables that need to be affected within a multiplicative factor of $c \log n$ is NP-hard for some positive $c$. On the positive side, we 
show it is possible to find sets of variables matching this inapproximability barrier in polynomial time. This can be done by a simple
greedy heuristic which sequentially picks variables to maximize the rank increase of the controllability matrix. 
Experiments on Erdos-Renyi random graphs demonstrate this heuristic almost always succeeds at findings the 
minimum number of variables.
\end{abstract}

\section{Introduction}  This paper considers the problem of controlling large-scale systems where it is 
impractical or impossible to measure or affect more than a small number of  variables. We are motivated by 
the recent emergence of control-theoretic ideas in the analysis of of biological circuits \cite{M-et-al-09},  biochemical
 reaction networks \cite{LSB13}, and systems biology \cite{R-12}, where systems of many reactants
often need to be driven to desirable states with inputs which influence only  a few key locations within the system. 


While controllability and observability aspects of linear and nonlinear systems have been amply studied 
since the introduction of the concept by Kalman \cite{K63}, it appears that with the exception relatively recent work 
little attention has historically been paid to the interaction of 
controllability 
with sparsity. 
There has been much contemporary interest in the control of large networks which have a nearly unlimited number of interacting parts -
from the aforementioned biological systems, to the smart grid \cite{Z09,CI12}, to traffic systems \cite{CC10, CHW11} - leading to the emergence of a budding theory of control over 
networks (see the \cite{OS07, MS08} and the references therein). The enormous size of these systems makes it costly to affect them
with inputs which feed into a nontrivial fraction of the states; similarly, observing the system becomes uniquely challenging since it becomes impractical to measure more than a negligible fraction of the 
variables. These issues  motivate the  study of control  strategies which interact sparsely with systems by measuring and affecting only a small number of  variables. 

We thus consider the following problem: given a linear time-invariant system \[ \dot{x}_i(t)  = \sum_{j = 1}^n a_{ij} x_j(t), ~~~~~ i = 1, \ldots, n, \] where the matrix $A =[a_{ij}]$ is known, find \\  \smallskip (i) a subset ${\cal O}$ of the variables $\{ x_1, \ldots, x_n \} $ to observe which is as small as possible; \\ (ii) a subset of the 
variables ${\cal I}$ to affect with an input which is as small as possible;\\ \smallskip so that the resulting system
\begin{eqnarray*} \dot{x}_i(t)  & =& \sum_{j=1}^n a_{ij} x_j(t)  + u_i(t), ~~~~ i \in {\cal I} \\
 \dot{x}_i(t)  & = & \sum_{j=1}^n a_{ij} x_j(t), ~~~~~~~~~~~~~ i \notin {\cal I} \\ 
y_i(t) & = & x_i(t), ~~~~~~~~~~~~~~~~~~~~~~ i \in {\cal O} 
\end{eqnarray*} is both controllable and observable. 


Due to the duality between controllability and observability, the  problems of looking for the sets ${\cal I}$ and ${\cal O}$ are equivalent to each other and may be considered independently. Thus, to simplify 
matters, we will just consider  the problem of devising an algorithm for finding a set ${\cal I}$ as small as possible which makes the system controllable. In matrix form, we are seeking to find a matrix $B$ which is {\em diagonal} and has as few nonzero entries as possible so that the system  
\begin{equation} \dot{x} = Ax + Bu \label{system} \end{equation}  is controllable. We will also consider a variation on this problem wherein we search for a {\em vector} $b$ which has as few nonzero entries
as possible so that 
the system 
\begin{equation} \dot{x} = Ax + bu \label{vectorsystem} \end{equation} is controllable.  Note that in Eq. (\ref{system}), $u$ is a vector of inputs, 
while in Eq. (\ref{vectorsystem}), $u$ is a scalar. We will refer to all these questions as 
minimal controllability problems\footnote{One may also consider variations in which we search for a matrix $B' \in \R^{n \times n}$ rendering $\dot{x}=Ax+B'u$ controllable
while seeking to minimize either the number of nonzero entries of $B'$, or the number of rows of $B'$ with a nonzero entry (representing the number of
components of the system affected). However, for any such matrix $B'$, we can easily construct a diagonal matrix $B$ rendering the system controllable without increasing the number of nonzero entries, or the number of rows with a nonzero entry: indeed, if the $i$'th row of $B'$ contains a nonzero entry, we simply set $B_{ii}=1$, and else we set $B_{ii}=0$. Consequently, both of these variations are easily seen to be equivalent to the problem of finding the sparsest diagonal matrix. }. 



 Our work follows a number of recent papers on the controllability properties of networks \cite{CMM12, EMCCB12, JLY12, LCWX08, LSB11, LSB13, MEB, nv12, pn-12, pljb12, RJME09, SH13, t04, Motter1, Motter2, pendulum, schuppen, cao-recent} with
similar motivations. These papers sought to find 
connections between the combinatorial properties of graphs corresponding to certain structured matrices $A$
and vectors $b$ and the controllability of Eq. (\ref{system}) and Eq. (\ref{vectorsystem}), either in the direct sense as we consider it or in a structural sense. We refer the interested reader to the recent survey
\cite{EMCCB12} for an overview of this literature. We note that the existence of easily optimizable necessary and sufficient combinatorial conditions for controllability would, 
of course, lead to a solution of the minimal controllability problem we have proposed here; however, to date it is only for matrices $A$ corresponding to 
influence processes on graphs like paths and cycles and grids and cartesian products
that such conditions have been found \cite{CMM12,pn-12, np-recent}.

We mention especially the recent paper \cite{LSB11} as well as the papers \cite{cm, pequito}, which have the a similar starting point as the present paper: finding effective algorithms for sparsely controlling linear systems. Our papers differs from those works primarily in that they consider the problem of making linear systems structurally controllable or strongly structurally controllable, whereas we consider the problem of making them controllable  (in the plain sense). In addition,  we mention that we measure sparsity by the number of nonzero variables of the systems's state vector $x$ we are affecting with the input; by contrast, \cite{LSB11} measures sparsity by counting the number of  driver nodes (see \cite{LSB11} for a definition).

Our paper has two main results. The first of these demonstrates the intractability of the minimal controllability problems we have 
defined and shows
that these problems are even intractable to approximate.  Let us 
adopt the notation ${\cal B}$ for
the set of vectors in $\R^n$ which make Eq. (\ref{vectorsystem}) controllable and ${\cal B}'$ for the set of diagonal matrices in $\R^{n \times n}$ which make Eq. (\ref{system}) controllable. 
We then have the following theorem.

\begin{theorem}  Approximating the number of nonzero entries of the sparsest
vector in ${\cal B}$ within  a multiplicative factor of $c \log n$ is NP-hard for some absolute constant\footnote{An absolute constant is a constant
which does not depend on any of the problem parameters, i.e., in this case it does not depend on $A$ or $n$.} $c>0$; moreover, this 
remains the case even  under the 
additional assumption that the matrix $A$ is symmetric. The same is true for the problem of approximating the 
number of nonzero entries of the sparsest diagonal matrix in ${\cal B}'$. \label{th:hardness}
\end{theorem} 

Thus not only is it impossible to compute  in polynomial time the sparsest $b \in {\cal B}$ (or $B \in {\cal B}'$) unless $P=NP$; it is NP-hard to even approximate the number of nonzero entries in the sparsest elements in ${\cal B}$ and ${\cal B}'$ within a multiplicative logarithmic factor.  

At first glance, this is a quite pessimistic result, but we next argue the situation is not as bad as it appears. Indeed,
from the perspective of designing scalable control laws, we don't necessarily need the sparsest input. In terms of our motivating scenario, the important thing is to avoid the worst-case where we use an input which 
affects roughly as many variables as the total number of variables $n$,  which may very well be enormous. 

In short, we need to find  inputs  which control a system of $n$ variables by affecting it in a negligible 
fraction of the variables. Thus we consider to question of finding $b \in {\cal B}$ and $B \in {\cal B}'$ whose number of nonzero 
entries is small relative to the total number of variables $n$. 

Unfortunately, the previous theorem rules out the possibility of finding a vector $b \in {\cal B}$  which has fewer nonzero entries than $c \log n$ times the sparsest vector in ${\cal B}$ for some $c>0$; the same impossibility result holds for the problem of approximating the minimally sparse diagonal 
matrix in ${\cal B}'$.  Nevertheless, we observe that for the purposes of 
controlling large systems, a $c \log n$ approximation to the sparest input will often be sufficient. Indeed, supposing that the sparsest input feeds into 
$p(n)$ components of the vector $x$, a $c \log n$ approximation will feed into at most $c p(n) \log n$ entries. If the system is controllable from a sufficiently small number of entries - specifically, if\footnote{The notation $f(n) = o(g(n))$ means that $\lim_{n \rightarrow \infty} \frac{f(n)}{g(n)} = 0$.}  $p(n) = o( n / \log n)$ - then the resulting control strategy renders a system with $n$ variables controllable from 
an asymptotically negligible $o(n)$ number of inputs. 

Fortunately, finding  $c \log n$-approximately-sparse elements of ${\cal B}$ and ${\cal B'}$ can be done in polynomial time, as our next theorem demonstrates.

\begin{theorem} There exists an algorithm whose running time is polynomial in $n$ and the number of bits in the entries
of $A$ which, under the assumption ${\cal B}$ is nonempty, returns an element in ${\cal B}$  whose number 
of nonzero entries is at most $c' \log n$ times the sparsest element in ${\cal B}$ for some absolute constant $c'>0$. Moreover, the 
same statement is true with ${\cal B}$ replaced by ${\cal B}'$. \label{th:alg}
\end{theorem} 

As we will  show, the algorithms which achieve this are simple greedy heuristics which sequentially add nonzero entries to 
maximize the rank increase of the controllability matrix. A central finding of this work, therefore, is that the performance guarantees of such 
simple heuristics cannot, in the most general case, be significantly improved. 

 Moreover, our simulation results on matrices $A$ deriving from  Erdos-Renyi random graphs show that we can usually finds a vector in ${\cal B}$ with only a very small number of nonzero entries: all such grandom graphs we simulated were found to be controllable by a greedy heuristic from one or two nodes. 

We now outline the remainder of the paper. Section \ref{sec:hardness} contains the proof of 
Theorem \ref{th:hardness} and a few of its natural corollaries (for example, we consider the related problem of rendering a linear system both controllable and
observable by measuring and affecting the smallest number of variables). Section \ref{sec:alg} contains the proof of Theorem \ref{th:alg}. We
report on the results of experiments with Erdos-Renyi random graphs in Section \ref{simul} and some conclusions are drawn in Section \ref{sec:concl}.

\section{Intractability results for minimal controllability problems\label{sec:hardness}}

This section is dedicated to the proof of Theorem \ref{th:hardness} and its consequences. We will break the proof of 
this theorem into two parts for simplicity: first we will show that the minimal controllability problem is NP-hard for a general
matrix $A$, and then we will argue that the proof can be extended to symmetric matrices.  

The proof itself proceeds by reduction to the hitting set problem, defined next. 

\begin{definition} Given a collection ${\cal C}$ of subsets of \{1, \ldots, m\}, the minimum hitting set problem 
asks for a set of smallest cardinality that has nonempty intersection with each set in ${\cal C}$. 
\end{definition}

The minimum hitting set problem is NP-hard and moreover is NP-hard to find a set whose cardinality is
within a factor of $c \log n$ of the optimal set, for some $c>0$ (see \cite{mos} for this hardness result for the 
set cover problem, easily seen to be equivalent). Moreover, it is easy to see that we can make 
a few natural assumptions while preserving the NP-hardness of the hitting set problem: we can assume that each set in ${\cal C}$ is nonempty and we can assume that every element appears in at least one set. We will argue that hitting set may be encoded in minimal
controllability, so that the latter must be NP-hard as well. We next give an informal overview of our argument. 

The argument is an application of the PBH test for controllability  which tells us that 
to make Eq. (\ref{system}) controllable, it is necessary and sufficient to choose a vector $b$ that is not orthogonal to all the left-eigenvectors of the matrix $A$. It is easy to see that if $A$ does not have any repeated eigenvalues, this is possible if and only if the support of $b$ intersects
with the support of every left-eigenvector of $A$. Defining ${\cal C}$ to be the collection of supports of the eigenvectors, this
turns into a minimum hitting set problem. 

However, the resulting hitting set problem has some structure, so that we cannot conclude that minimal controllability is
NP-hard just yet; for example, clearly 
the sets in the collection ${\cal C}$ defined in the previous paragraph are the supports of the columns of an invertible matrix, so
they cannot be arbitrarily chosen. In the case where $A$ is symmetric and its 
eigenvectors orthogonal, there is more structure still. The challenge is therefore to encode arbitrary hitting set 
problems within this structure. 

 We now proceed to the formal argument. First, however, we introduce some notation which we will use throughout the
remainder of this paper. We will say that a matrix or a vector is $k$-sparse if it 
has at most $k$ nonzero entries, and we will say that the matrix $A$ is $k$-controllable if there exists a vector $b \in {\cal B}$ which is $k$-sparse which makes Eq. (\ref{system}) controllable. We will sometimes say that a vector $b$ makes $A$ controllable, meaning that
Eq. (\ref{system}) is controllable with this particular $A,b$. We will use ${\bf I}_k$ to denote the $k \times k$ identity matrix,  ${\bf 0}_{k \times l}$ denote the $k \times l$ zero matrix, and  ${\bf e}_{k \times l}$ to mean the $k \times l$ all-ones matrix. Finally, ${\bf e}_i$ will
denote the $i$'th basis vector.

\begin{proof}[Proof of Theorem \ref{th:hardness}, First part] We will first prove the intractability of finding the sparsest vector in ${\cal B}$ without the
assumption that $A$ is symmetric. 

Given a collection ${\cal C}$  of $p$ nonempty subsets of $\{1, \ldots, m\}$ such that
every element appears in at least one set,  we define the incidence matrix $C \in \R^{p \times m}$ where we set $C_{ij} = 1$ if the $i$'th set contains the element $j$, and zero
otherwise. We then define the matrix $V$ in $\R^{(p+m+1) \times (p+m+1)}$ as
\[ V = \begin{pmatrix} 2 {\bf I}_m & {\bf 0}_{m \times p} & {\bf e}_{m \times 1} \\ C & (m+1){\bf I}_p & {\bf 0}_{p \times 1} \\ {\bf 0}_{1 \times m} & {\bf 0}_{1 \times p} & 1 \end{pmatrix} \] By construction, $V$ is strictly  diagonally dominant, and therefore invertible. We set 
\[ A({\cal C}) = V^{-1} {\rm diag}(1, \ldots, m+p+1 )  V \] Note that $A({\cal C})$ has distinct eigenvalues and its left-eigenvectors are the 
rows of $V$. Moreover $A$ may be constructed in polynomial-time from the collection ${\cal C}$. Indeed,  since
every element appears in at least one set,  we have that it takes at least $\max(m,p)$ bits to describe the set collection ${\cal C}$, and  
inverting $V$ and performing the necessary multiplications to construct $A$ can be done in polynomial time in $m,p$ \cite{S98}. 

{\em We claim that ${\cal C}$ has a hitting set with cardinality $k$ if and only if 
$A({\cal C})$ is $k+1$-controllable. }

Indeed, suppose that a hitting set $S'$ of cardinality $k$ exists for ${\cal C}$. Set $b_i=1$ for all $i \in S'$ and $b_{m+p+1}=1$; set all other entries of $b$ 
to zero. By construction, $b$ has positive inner product with every row of $V$. The PBH
test then implies that the system of Eq. \eqref{system} with this $b$ is controllable. 

Conversely, suppose  the system is controllable with a $k+1$-sparse vector $b$. Once again we appeal to the 
PBH condition, which gives that $b$ is not orthogonal to any row of $V$. Since $V$ is nonnegative, we 
may take the absolute value of every entry of  $b$: the PBH test implies this operation preserves the controllability of 
Eq. \eqref{system}. Moreover,  the PBH condition for $A({\cal C})$ and nonnegative $b$  is that
for any row $i$ of $V$, there is some index $j$ such that both $V_{ij}$ and $b_j$ are positive. We will refer to this as the 
intersection property. 

By considering the last row of $V$, we see that the intersection property immediately implies that $b_{m+p+1}$ is positive. We now argue that we can find a $k+1$-sparse vector $b$ whose
support is contained in $\{1, \ldots, m \} \cup \{m+p+1\}$ such that the system of Eq. (\ref{system}) is controllable. Indeed, if $b_i > 0$ for some $i \notin  \{1, \ldots, m \} \cup \{m+p+1\}$ then
we modify $b$ by setting first $b_i=0$ and then $b_l =1$ for any index $l \in \{1, \ldots, m \} $ such that $V_{il}>0$, i.e. any $l$ belonging to the 
$i$'th set of ${\cal C}$. In so doing, we preserve the controllability of Eq. \eqref{system} since, by construction, for $i \notin  \{1, \ldots, m \} \cup \{m+p+1\}$ no row besides $i$ has $i$'th entry positive, so the intersection condition still holds. Moreover, the vector $b$ certainly remains $k+1$-sparse. 

Proceeding in this way for every $i \notin  \{1, \ldots, m \} \cup \{m+p+1\}$ we finally get a $k+1$-sparse vector $b$ whose support is contained
in $ \{1, \ldots, m \} \cup \{m+p+1\}$ and which makes Eq. \eqref{system} controllable. The intersection property then implies that for each $i=1, \ldots, p$ there is some index $j$ such that
$b_j > 0$ and the $m+i$'th row of $V$ is positive on the $j$'th entry. By construction, this means that ${\rm supp}(b) \cap \{1, \ldots, m\}$ is a 
hitting set for ${\cal C}$. Since $b$ is $k+1$-sparse and $b_{m+p+1}>0$, this hitting set has cardinality at most $k$. 

This concludes the proof of the italicized claim above that the solution of the minimal controllability problem is one more than the
size of the smallest hitting set for ${\cal C}$. 

Thus, if we could approximate the minimal controllability problem for $A \in \R^{n \times n}$ to a multiplicative factor of $c \log n$ for some $c>0$ in polynomial time, we could achieve a $c \log (m+p+1)$-multiplicative approximation to hitting set on $p$ subsets of $\{1, \ldots, m\}$. But, due to known results on the hardness of hitting set \cite{mos}, for some values of $c$ this would imply $P=NP$. 
 \end{proof}

We have thus proved the intractability of finding the sparsest vector in ${\cal B}$, without the assumption that $A$ is symmetric. We next 
proceed to argue that the problem of finding the sparsest matrix in ${\cal B}'$ is intractable as well, still without the assumption of symmetry of 
$A$. This will follow fairly straightforwardly out of the following lemma. 

\smallskip

\begin{lemma} If $A$ has distinct eigenvalues, then the number of nonzero entries in the sparsest vector in ${\cal B}$ is the same as the number 
of nonzero entries in the sparsest matrix in ${\cal B}'$. \label{equiv}
\end{lemma} 

\smallskip

\begin{proof}  For one direction, observe that any $b \in {\cal B}$ can be made into a diagonal matrix $B$ by putting the entries of $b$ as diagonal 
entries. The PBH theorem implies that for every left-eigenvector $v$ of $A$, there exists at least one index $j$ with $v_j \neq 0, b_j \neq 0$, and 
consequently there will be a column of $B$ to which $v$ is not orthogonal, implying that $B \in {\cal B'}$. We have thus proved that if a $b \in {\cal B}$ is $k$-sparse,
then we can find a $k$-sparse $B \in {\cal B}'$. 

For the other direction, given a $B \in {\cal B}'$, we can come up with a vector $b$ by setting $b_j$ to an i.i.d. standard Gaussian whenever
$B_{jj} \neq 0$. Since for every left-eigenvector $v$ of $A$, there is some index $j$ with $v_j \neq 0, b_j \neq 0$, we have that  $v^T b$ is a linear combination of continuous random variables with at least one positive coefficient, and consequently it is not equal to zero with probability $1$. 
Appealing now to the fact that $A$ has distinct eigenvalues, and consequently $n$ distinct left-eigenvectors after normalization, we get that with probability 
$1$ no left-eigenvector of $A$ is orthogonal to $b$. This implies that if there exists $B \in {\cal B}'$ which is $k$-sparse, then there exists at least one $k$-sparse $b \in {\cal B}$.  
\end{proof}

\smallskip

\begin{proof}[Proof of Theorem \ref{th:hardness}, Second part] We now argue that finding the sparsest matrix in ${\cal B}'$ is intractable. Indeed, since the matrix $A({\cal C})$ constructed the first part of the proof has distinct eigenvalues, we have that Lemma \ref{equiv} implies that the sparsest vector in ${\cal B}$ for this $A$ has the same number of nonzero entries as the sparsest matrix in ${\cal B}'$. Consequently, the intractability guarantees of hitting set extend to the problem of finding the sparsest matrix in ${\cal B}'$.  
\end{proof}

\smallskip 

This proves Theorem \ref{th:hardness} without the additional assumption that $A$ is a symmetric matrix. Before proceeding
to describe the proof in the symmetric case, we illustrate the construction with a simple example. 

\smallskip

\noindent {\bf Example:} Suppose $${\cal C} =  \{1,2\}, \{2,3\}, \{1,3\}, \{1,2,3\}.$$  Here $m=3$ and $p=4$. The smallest hitting set is of size $2$, but there is no hitting set of size $1$. The incidence matrix is 
\[ C = \begin{pmatrix} 
1 & 1 & 0 \\
0 & 1 & 1 \\
1 & 0 & 1 \\
1 & 1 & 1  
\end{pmatrix} \] and \[ V =  \begin{pmatrix} 
2 & 0 & 0 & 0 & 0 & 0 & 0 & 1 \\
0 & 2 & 0 & 0 & 0 & 0 & 0 & 1 \\
0 & 0 & 2 & 0 & 0 & 0 & 0 & 1  \\
1 & 1 & 0 & 4 & 0 & 0 & 0 & 0 \\
0 & 1 & 1 & 0 & 4 & 0 & 0 & 0 \\
1 & 0 & 1 & 0 & 0 & 4 & 0 & 0 \\
1 & 1 & 1 & 0 & 0 & 0 & 4 & 0 \\
0 & 0 & 0 & 0 & 0 & 0 & 0 & 1 
\end{pmatrix} \] and  $A = V^{-1} {\rm diag}(1,2,3,4,5,6,7,8) V$ is 
\[ A =  \begin{pmatrix} 
1 & 0 & 0 & 0 & 0 & 0 & 0 & -7/2 \\
0 & 2 & 0 & 0 & 0 & 0 & 0 & -3 \\
0 & 0 & 3 & 0 & 0 & 0 & 0 & -5/2  \\
3/4 & 1/2 & 0 & 4 & 0 & 0 & 0 & 13/8 \\
0 & 3/4 & 1/2 & 0 & 5 & 0 & 0 & 11/8 \\
5/4 & 0 & 3/4 & 0 & 0 & 6 & 0 & 3/2 \\
3/2 & 5/4 & 1 & 0 & 0 & 0 & 7 & 9/4 \\
0 & 0 & 0 & 0 & 0 & 0 & 0 & 8 
\end{pmatrix} \]  Direct computation shows that choosing $b = \begin{pmatrix} 1 & 1 & 0 & 0 & 0 & 0 & 0 &  1 \end{pmatrix}^T$ makes
the system controllable. Moreover,  since it is easy to see that no hitting set for ${\cal C}$ of size one exists, and since in the course of our proof we have shown that the existence of a hitting set of size $k$ for ${\cal C}$ is equivalent to $k+1$ controllability of $A$, we know that no vector sparser than $b$ renders $A$ controllable. 

We now proceed with the proof of Theorem \ref{th:hardness}, which remains to be shown in the case the matrix is
symmetric. For this we will need the following lemma, which is a minor modification of known facts about the 
Gram-Schmidt process.

\begin{lemma} \label{extension} Given a collection of $k \geq 1$ orthogonal vectors $v_1, \ldots, v_k$ all in ${\bf e}_1^\perp$ in $\R^n$, whose 
entries are rational with bit-size $B$, 
it is possible to find in polynomial time in $n,B$ vectors $v_{k+1}, \ldots, v_n$ such that: (i) $v_1, \ldots, v_n$ is an orthogonal
collection of vectors (ii) the first component of each of the vectors $v_{k+1}, \ldots, v_n$ is nonzero. 
\end{lemma}

\begin{proof} Set $v_{k+1}={\bf e}_1$ and use Gram-Schmidt to find $v_{k+2}, \ldots, v_n$ so that 
$v_1, \ldots, v_n$ is an orthogonal basis for $\R^n$. Now let ${\cal S}$ be the set of those vectors among $v_{k+1}, \ldots, v_n$ which 
have zero first entry; clearly we have ${\cal S}=\{k+2, k+3, \ldots, n\}$ since $v_{k+1}$ has first entry of $1$ and zeros in all other entries and $v_{k+2}, \ldots, v_n$ are orthogonal to $v_{k+1}$.

Now given a collection $v_1, \ldots, v_n$ such that $v_{k+1}$ has first entry of $1$, we describe an operation
that decreases the size of ${\cal S}$.  Pick any $v_l \in {\cal S}$ and update as
\begin{eqnarray*} v_l & \leftarrow & c v_l + v_{k+1} \\ 
v_{k+1} & \leftarrow & - v_l + v_{k+1} 
\end{eqnarray*} where $c = ||v_{k+1}||_2^2/||v_l||_2^2$. Observe that the collection $v_1, \ldots, v_n$ remains orthogonal after this update and that the new
vectors $v_{k+1}, v_l$ both have $1$ as their first entry. 

As long as there remains a vector in $S$, the above operation decreases the cardinality of ${\cal S}$ by $1$. It follows that
we can make ${\cal S}$ empty, which is what we needed to show. It is easy to see that the procedure takes polynomial-time
in $n,B$. 
\end{proof}

We now complete the proof of Theorem \ref{th:hardness}.

\begin{proof}[Proof of Theorem \ref{th:hardness}, Third part] Let us introduce the notation ${\rm mc}(X)$ to mean the
number of nonzero entries in a solution $b$ of the minimal controllability problem with the matrix $X$. We will next construct, 
in polynomial time, a symmetric matrix $\widehat A$ with distinct eigenvalues and prove that it satisfies the inequality ${\rm mc}(\widehat A)/{\rm mc}(A) \in [1/3,2]$ (where $A$ is the matrix constructed in the first part of the proof).
Once this has been established we will then have that since approximating ${\rm mc}(A)$ 
to within a multiplicative factor of $c \log n$ is NP-hard for some $c>0$, the same holds for ${\rm mc}(\widehat A)$.  Due to Lemma \ref{equiv}, this
also implies the corresponding hardness of approximating the sparsest diagonal matrix which makes $\widehat A$ controllable. 

We will construct $\widehat A$ as follows: we will take $V$ and add ${m+p+1 \choose 2}+1$ columns to it; then we will add a corresponding number of rows
to make the resulting matrix square.  We will call the resulting matrix $\widehat V$; we have that $\widehat V \in \R^{r \times r}$ where $r = m+p+2 + {m + p + 1 \choose 2}$.

We now describe how we fill in the extra entries of $\widehat V$. Let us index the first ${m+p+1 \choose 2}$ additional columns by $\{i,j\}$ for $i,j = 1, \ldots, m+p+1$. If the $i$'th and $j$'th rows of $V$ have nonzero inner product, we set one of $\widehat V_{i, \{i,j\}}, \widehat V_{j, \{i,j\}}$ to $1$ and the other to 
the negative of the inner product with the $i$'th and $j$'th rows of $V$; else, we set both of them to zero. All other additional entries in the first 
$m+p+1$ rows of $\widehat V$ will be set to zero. Note that by construction the first $m+p+1$ rows of $\widehat V$ are orthogonal.

As for the extra rows of $\widehat V$, we will fill them in using Lemma \ref{extension} to be orthogonal to the existing rows of
$V$ and have a nonzero entry in the $r$'th coordinate. By construction the rows of  $\widehat V$ are orthogonal.  Finally, $\widehat A = \widehat V^{-1} {\rm diag}(1,2,\ldots,r) \widehat V$. Note that $A$ is symmetric and the  the left-eigenvectors of $\widehat A$ are the rows of $\widehat V$. 

Now that we have constructed $\widehat A$, we turn to the proof of the assertion that ${\rm mc}(\widehat A)/{\rm mc}(A) \in [1/3,2]$. 
Indeed, suppose $A$ is $k$-controllable, i.e., there exists a $k$-sparse vector $b \in \R^{m+p+1}$ which makes Eq. (\ref{system}) 
controllable. We then define a $k+1$-sparse vector $\widehat b \in \R^r$ by setting its 
first $m+p+1$ entries to the entries of $b$ and setting its $r$'th entry to a random number generated according to 
any continuous probability distribution, say a standard Gaussian distribution; the rest of the entries
of $\widehat b$ are zero. By construction, rows $1, \ldots, m+p+1$ of $\widehat V$ are not orthogonal to $b$; and 
the probability that any other row is orthogonal to $\widehat b$ is zero. We have thus generated a $k+1$-sparse vector
$\widehat b$ which makes $\widehat A$ controllable with probability $1$. We thus conclude that there exists a $k+1$
sparse vector which makes $\widehat A$ controllable, for if such a vector did not exist, the probability that $\widehat b$ makes
$\widehat A$ controllable would be zero, not one. Finally, noting that $(k+1)/k \leq 2$, we conclude that ${\rm mc}(\widehat A) \leq 2 {\rm mc}(A)$. 

For the other direction, suppose now that $\widehat{A}$ is controllable with a $k$-sparse vector $\widehat b$. We argue that there exists 
a $3k$ sparse vector $b'$ making $\widehat A$ controllable which has the property that $b_j'=0$ for all $j \notin \{1, \ldots, m+p+1\} \cup \{ r \}$.

Indeed, we now describe how we can manipulate $\widehat b$ to obtain $b'$. Let ${\cal S}$ be the support of $\widehat b$. If some element $j$ not in
$\{1, \ldots, m+p+1\} \cup \{ r \}$ is in ${\cal S}$, we remove it by setting $b_j=0$; now observe that at most two of 
the first $m+p+1$ rows of $ \widehat V$
have a nonzero entry in the $j$'th place; we add two elements   in $\{1, \ldots, m+p+1\}$ to ${\cal S}$, one from the support of each of those two rows within $\{1, \ldots, m+p+1\}$,  by setting the corresponding entries of $b$ to a continuous positive random variable. Specifically, suppose that $j \in {\cal S}$ for some $j \notin \{1, \ldots, m+p+1\} \cup \{r\}$, and   $\widehat V_{i_1,j}>0$ and $\widehat V_{i_2,j}>0$ for $i_1, i_2 \in \{1, \ldots, m+p+1\}$ are the two rows among the first $m+p+1$ with a nonzero coordinate in the $j$'th spot; we then pick any index $k_1 \in \{1, \ldots, m+p+1\}$ in the support of the $i_1$'st row of $\widehat V$ and index $k_2 \in \{1, \ldots, m+p+1\}$ in the support of the $i_2$'nd row of $\widehat V$ and  we set 
set $b_{k_1}$ and $b_{k_2}$ to be, say,  independent uniform random variables on $[1,2]$. Finally, we add
$r$ to ${\cal S}$ by setting $b_r$ to a positive continuous random variable, to make sure that the rows $m+p+2, \ldots, r$ have supports which overlap with the support of $b$. 

Proceeding this way, we end up with a support ${\cal S}$ which is at most three times as many elements as it had initially - because each time we remove an index we add at most three indices -  but which has no element outside of $\{1, \ldots, m+p+1\} \cup \{ r \}$.  Finally, we note that for any row of $\widehat V$, the probability that it is orthogonal to $b'$ is zero. Thus with probability one, the vector
$b'$ makes $\widehat A$ controllable. Consequently, we conclude there exists a 
$3k$ sparse vector that makes $\widehat A$ controllable.

It now follows that the vector $b \in \R^{m+p+1}$ obtained by taking the first $m+p+1$ entries of $b'$ makes $A$ controllable. Thus
${\rm mc}(A) \leq 3 {\rm mc} (\widehat A)$.  This concludes the proof.
\end{proof}

We now proceed to state and prove some natural corollaries of Theorem \ref{th:hardness}. First, we note that altering the dimension of $B$ does not make a difference .

\begin{corollary} For any $l \geq 0$, finding the sparsest $B \in \R^{n \times l}$ such that Eq. (\ref{system}) is controllable is NP-hard.  
\end{corollary} 

\begin{proof} We argue that for $A$ with distinct eigenvalues this is identical to the problem of finding a vector $b$ making the system controllable. Since we have already made a similar argument in the proof of the third part of Theorem \ref{th:hardness}, we will only sketch the argument here. Indeed, if $A$ is a $k$-controllable
with a vector $b$, it is certainly $k$-controllable with a matrix $B \in \R^{n \times l}$ - one can just put this vector as the first column of the matrix. Conversely, if such an $A$ is $k$-controllable with a matrix $B \in \R^{n \times l}$, it is $k$-controllable with a vector $b$ obtained by setting $b_i$ to a standard Gaussian whenever $B_{ij} \neq 0$ for some $j$.  
\end{proof}

\smallskip

A variation of the minimal controllability involves finding matrices $B$ and $C$ that make the system simultaneously controllable and
observable, with as few nonzero entries as possible. Unfortunately, the next corollary proves an intractability result in
this scenario as well. 

\begin{corollary} Finding two diagonal matrices $B$ and $C$ in $\R^{n \times n}$ with the smallest total number
of nonzero entries such that the 
system \begin{eqnarray} \dot{x} & = &  Ax + Bu \nonumber \\  y & =&  Cx  \label{obs_system} \end{eqnarray} is both controllable and observable 
is NP-hard. Similarly, for any $l_1, l_2 \geq 0$ finding two matrices $B \in \R^{n \times l_1}, C \in \R^{l_2 \times n}$ with the smallest total number of nonzero entries so that Eq. (\ref{obs_system})
is both controllable and observable is NP-hard. 
\end{corollary}

\begin{proof} Consider the matrix $V$ which was constructed in the proof of Theorem \ref{th:hardness}. We first argue that $V^{-1}$ has the same pattern of nonzero entries as $V$, with the exception of the final column of $V^{-1}$, 
which is nonzero in every entry. We will show this by explicitly computing $V^{-1}$. 

Indeed, for $i = 1, \ldots,m$ we have 
\[ (Vx)_i = 2x_i + x_{m+p+1} \] and moreover $(Vx)_{p+m+1}  = x_{p+m+1}$ so that the $i$'th row of $V^{-1}$ is immediate: it has a $1/2$ in the $i,i$'th place and a $-1/2$ in the $i,m+p+1$'th place, and zeros elsewhere.

Next for $i=m+1, \ldots, m+p$, we have 
\[ (Vx)_i = (m+1)x_i + \sum_{j \in C_i} x_j \] where $C_i$ is the corresponding set of the collection ${\cal C}$. Consequently, $V^{-1}$ has a 
$1/(m+1)$ in the $i,i$'th spot, a $-1/(2(m+1)$ in the $i,j$'th place where $j \in C_i$, and $|C_i|/(2(m+1))$ in the $i,m+p+1$'th place; every other entry 
of the $i'$h row is zero. Finally, the last row of $V^{-1}$ is clearly $[0, 0, \ldots, 0, 1]$. 

This explicit computation shows that $V^{-1}$ has the same pattern of nonzero entries as $V$, with the exception of the final column of $V^{-1}$ which
is nonzero in every entry. Now for the system of Eq. (\ref{obs_system}) to be observable, we need to satisfy the PBH conditions, namely
that each column of $V^{-1}$ is not orthogonal to some row of $C$. By considering columns $m+1, \ldots, m+p$, we immediately see that 
$C$ needs to have at least $p$ nonzero entries; moreover, it is  easy to see that $p$ entries always suffice, either under the condition that $C \in \R^{n \times n}$ is diagonal or that its dimensions are $l_2 \times n$. 

Thus the problem of minimizing the number of nonzero entries in both $B$ and $C$ to make Eq. (\ref{obs_system}) both observable 
and controllable reduces to finding the sparsest $B$ to make it controllable. We have already shown that this problem is NP-hard. 
\end{proof}

\section{Approximating minimal controllability\label{sec:alg}}

 We will now consider the problem of controlling a linear system with an input which does not affect too many variables in the
system.  Our goal in this section will be to prove Theorem \ref{th:alg}, showing that we can match
the lower bounds of Theorem \ref{th:hardness} by returning a $c' \log n$ approximate solution for the minimal controllability problems of
finding the sparsest elements of ${\cal B}$ and ${\cal B}'$. 

 These approximation guarantees will be achieved by simple greedy heuristics 
which repeatedly add nonzero elements to the input to maximize the rank increase for the controllability matrix. The motivation
lies in the standard greedy algorithm for hitting set, which also achieves a multiplicative logarithmic approximation by  repeatedly adding elements which
appear in as many of the remaining sets as possible. Since the previous section drew on connections between  hitting sets and  controllability,  it is reasonable to guess that a similar techniques might achieve a comparable bound for minimal controllability.

Let us adopt the following conventions. We will refer to the problem of finding the sparsest element in ${\cal B}$ as the minimal vector 
controllability problem. We will refer to the problem of finding the sparsest element in ${\cal B}'$ as the minimal diagonal controllability 
problem. 

\subsection{The minimal vector controllability problem\label{vector}} 

For simplicity, let us adopt the notation $|b|$ for the number of nonzero entries in the vector $b$ (and similarly, for matrices, $|B|$ is 
the number of nonzero entries of the matrix $B$). Our first step will be
to describe a randomized algorithm for the minimal vector controllability problem which assumes the ability to set entries of the
vector $b$
to Gaussian random numbers. This is not a plausible assumption: in the real-world, we can generate random bits, rather than
infinite-precision random numbers. Nevertheless, this assumption considerably simplifies the 
proof and clarifies the main ideas involved. After giving an algorithm for approximating minimal controllability
with this unrealistic assumption in place, we will show how to modify this algorithm to work without this assumption. The final
result is a  a deterministic polynomial-time algorithm for $c'\log n$-approximate controllability.

We begin with a statement of our randomized algorithm in the box below (we will refer to this algorithm
henceforth as {\em the randomized algorithm}). The algorithm has a very simple
structure: given a current vector $b$ it tries to update it by setting one of the zero entries to an independently
generated standard Gaussian. It picks the index which maximizes the increase in the rank of the 
controllability matrix $C(A,b) = \begin{pmatrix} b & Ab & A^2 b & \cdots & A^{n-1} b \end{pmatrix}$. 
When it is no longer able to increase the rank of $C(A,b)$ by doing so, it stops.

\begin{algorithm} \begin{enumerate} \item Initialize $b$ to the zero vector and set $c^* = 1$.
\item {\bf While} $c^*>0$,
\begin{enumerate} 

\item {\bf For} $j=1, \ldots, n$, 
\begin{enumerate} \item If $b_j \neq 0$, set $c(j) = 0$. \item  If $b_j = 0$, set $\widetilde{b} = b + X(j) {\bf e}_j$ where $X(j)$ is i.i.d. standard normal,
and set $c(j) = {\rm rank}(C(A, \widetilde b)) - {\rm rank}(C(A,b))$. 
\end{enumerate}
$~${\bf end}
\item Pick $j^* \in \arg \max_j c(j)$ and set $c^* = c(j^*)$. 
\item If $c^*>0$, set $b \leftarrow b + X(j^*) {\bf e_{j^*}}$.
 \end{enumerate} 
$~${\bf end}
\item Output $b$. 
\end{enumerate}
\end{algorithm}

We will show that if ${\cal B}$ is nonempty, then
this algorithm returns a vector in ${\cal B}$ with probability $1$. Moreover, the number of nonzero entries in a vector 
returned by this algorithm approximates the sparsest vector in ${\cal B}$ by a multiplicative factor of $c' \log n$, for some
$c'>0$.

 {\em Without loss of generality, we will be assuming in the remainder of Section \ref{vector}
that each eigenspace of $A$ is one-dimensional}: if some eigenspace of $A$ has dimension larger than $1$ then  ${\cal B}$ is empty and all the results we will prove here about the minimal vector controllability problem (including Theorem \ref{th:alg}) are vacuously true. 

Let us adopt the notation $\lambda_1, \ldots, \lambda_k$ for the distinct eigenvalues of $A$. Let $T$ be a matrix 
that puts $A$ into Jordan form, i.e., $T^{-1} A T = J$, where $J$ is the Jordan form of $A$.  Due to the 
assumption that every eigenspace of $A$ is one-dimensional, we have that $J$ is block-diagonal with exactly $k$ blocks.

Without loss of generality, let us assume that the $i$'th block is associated with the eigenvalue $\lambda_i$. Moreover, let us denote
its dimension by $d_i$, and, finally, let us introduce the notation $t(i,j)$ to mean the 
$d_1 + d_2 + \cdots + d_{i-1} + j$'th row of $T^{-1}$. Observe that for fixed $i$, the collection $\{t(i,j) ~|~ j = 1, \ldots, d_i \}$ 
comprise the rows of $T^{-1}$ associated with the $i$'th block of $J$.  Naturally, the entire collection  $${\cal T} = \{ t(i,j),~ i=1, \ldots, k, ~j = 1, \ldots, d_i \}$$ is
a basis for $\R^n$. For a vector $v \in \R^n$ we use $v[i]$ to denote the vector in $\R^{d_i}$ whose $j$'th entry equals the
$d_1 + \cdots + d_{i-1} + j$'th entry of $v$. 
We will use the standard notation $\langle x, y \rangle = x^T y$ for the inner product of the vectors $x$ and $y$.
We will say that the vector $t(i,j) \in {\cal T}$ is {\em covered} by the vector $b$ if $\langle t(i,j'), b \rangle \neq 0$ for some $j' \leq d_i$ such that $j '\geq j$; else, we will say that $t(i,j)$ is {\em uncovered} by $b$. 

Our first lemma of this section provides a combinatorial interpretation of the rank of the controllability matrix in terms of the
number of covered vectors $t(i,j)$. This is a slight modification of the usual textbook characterization of controllability in terms of the Jordan form. While the lemma and its proof are quite elementary, we have been unable to find a clean reference and so we provide a proof to be
self-contained. Recall that we are assuming in this subsection that the eigenspaces of $A$ are all one-dimensional. 

\begin{lemma} \label{covered-rank} ${\rm rank}(C(A,b))$ equals the number of covered $t(i,j) \in {\cal T}$. 
\end{lemma}

\begin{proof} Let $U_d$ denote the upper-shift operator on $\R^d$ 
and let $J_i$ be the $i$'th diagonal block of $J$ so that $J_i = \lambda_i {\bf I}_{d_i} + U_{d_i}$. 
 Defining $b_i = (T^{-1}b)[i]$, we  have that  \begin{small}
\begin{eqnarray*} {\rm rank}~ C(A,b) & = & {\rm rank} ~C(T^{-1} A T, T^{-1} b) \\ 
& = & {\rm rank}~ C(J, T^{-1} b) \\ 
& = & \sum_{i=1}^k {\rm rank}~  \begin{pmatrix} b_i & J_i b_i  & J_i^2 b_i & \cdots & J_i^{d_i-1} b_i \end{pmatrix} \\ 
 \end{eqnarray*} \end{small} The first two equalities in the above sequence are obvious, while the last equality is justified as follows. 
Observe that the left-nullspace of $C(J, T^{-1} b)$ is comprised of all the vectors $q$ satisfying $q^T e^{tJ} T^{-1} b = 0$; this is proven in the usual way using the Cayley-Hamilton theorem in one direction and repeatedly differentiating the power series for $e^{tJ}$ in the other. The last expression can be rewritten as $\sum_{i=1}^k q[i]^T e^{t J_i} b_i = 0$. We can take the Laplace transform, which we know  exists for ${\rm Re}~s > \max_i {\rm Re} ~\lambda_i$,  to obtain that $\sum_{i=1}^k q[i]^T (sI - J_i)^{-1} b_i = 0$. Now
because different $J_i$ have different eigenvalues, we argue this is only possible if each term in the above sum is zero.  Indeed, consider what happens when $s$ approaches from the right an eigenvalue $\lambda_j$ with maximal real part.  If $q[j]^T (sI-J_j)^{-1} b_j$ is not identically zero for all $s \neq \lambda_j$ in the complex plane, we obtain that $q[j]^T (sI-J_j)^{-1} b_j$ blows up as $s$ approaches $\lambda_j$, while all other terms 
$q[i]^T (sI-J_i)^{-1} b_i$ remain bounded, which contradicts  $\sum_{i=1}^k q[i]^T (sI - J_i)^{-1} b_i = 0$. We thus have that $q[j]$ is in the left-nullspace of $e^{tJ_j} b_j$. By repeating this argument, we obtain that each $q[i]^T$ is in the left-nullspace of $C(J_i,b_i)$ for each $i=1, \ldots, k$,
and therefore the left-nullspace of $C(J,T^{-1} b)$ is a direct sum of the left-nullspaces of $C(J_i, b_i)$, which justifies the last equality above.

Next, since rank is unaffected by column operations, we have \begin{small}
\[ {\rm rank}(C(A,b)) = \sum_{i=1}^k {\rm rank} \begin{pmatrix} b_i & U_{d_i} b_i  & U_{d_i}^2 b_i & \cdots & U_{d_i}^{d_i-1} b_i \end{pmatrix} \] \end{small} But each term in this sum is just the largest index of a nonzero entry of the vector $b_i$. That is, 
letting $z_i$ be the largest $j \in \{ 1, \ldots, d_i \}$ such that the $j$'th entry of $b_i$ is nonzero, we get 
\[ {\rm rank}(C(A,b)) = \sum_{i=1}^k z_i \] This is exactly the number of covered elements in ${\cal T}$. 
\end{proof}

Our next lemma uses the combinatorial characterization we have just derived to obtain a performance bound
for the randomized algorithm we have described. 

\begin{lemma} \label{rand-lemma} Suppose that ${\cal B}$ is nonempty and let $b'$ be a vector in ${\cal B}$. Then with probability $1$ the randomized 
algorithm outputs 
a vector in ${\cal B}$ with $O(|b'| \log n)$ nonzero entries.
\end{lemma}

\begin{proof} Since $b' \in {\cal B}$, we know that every $t(i,j) \in {\cal T}$ is covered by $b'$. Thus, in particular, if ${\cal F}$ is any subset of ${\cal T}$, then every $t(i,j) \in {\cal F}$ 
is covered by $b'$. From this we claim that we can draw the following conclusion: for any ${\cal F} \subset {\cal T}$ there exist an index $j 
\in \{1, \ldots, n\}$ such that the choice $b=e_j$ covers at least $|{\cal F}|/|b'|$ of the vectors in ${\cal F}$. 

Indeed, if this were not true, then for every index $j$ the number of $t(i,j) \in {\cal F}$ such that some $t(i,j')$ with $j \leq j' \leq d_i$ has a nonzero
$j$'th entry is 
strictly fewer than than $|{\cal F}|/|b'|$. It would then follow that there is at least one $t(i,j) \in {\cal F}$ such that no $t(i,j'), j \leq j' \leq d_i$ has a nonzero entry
at any index where $b'$ is nonzero. But this contradicts every $t(i,j)$ being covered by $b'$.

Let us consider the $l$'th time when the randomized algorithm enters step 2.a. Let ${\cal F}(l)$ be the number 
of uncovered vectors in ${\cal T}$ at this stage. We have just shown that there exists some index $j$ such that at least $|{\cal F}(l)|/|b'|$ of 
the vectors in ${\cal F}(l)$ have nonzero $j$'th entry. With probability $1$, all of these vectors will become covered when 
we put an independent standard Gaussian in $j$'th entry, and also with probability $1$, no previously covered vector
becomes uncovered. Thus we can infer the following conclusions with probability $1$: first \begin{equation} 
\label{fdecrease} |{\cal F}(l+1)| \leq |{\cal F}(l)| - \frac{|{\cal F}(l)|}{|b'|}, \end{equation} since the index picked by the algorithm covers at least as many vectors as the index $j$; and also since $(1 - \frac{1}{x})^x \leq e^{-1}$ for all $x \geq 1$, we have that $|{\cal F}(l)|$ shrinks by a factor of at least $e^{-1}$ every $|b'|$
steps. After $O(|b'| \log n)$ steps, $|{\cal F}(l)|$ is strictly below $1$, so it must equal zero. This proves the lemma.  
\end{proof}

By choosing $b'$ to be the sparsest vector in ${\cal B}$ and applying the lemma we have just proved, we have the approximation guarantee that we
seek: the randomized algorithm finds a vector in ${\cal B}$ which is an $O(\log n)$ approximation to the sparsest vector. 

We next revise our  algorithm by removing the assumption that we can generate infinite-precision random numbers. The new
algorithm is given in a box below, and we will refer to it henceforth as {\em the deterministic algorithm}. 

\begin{algorithm} \begin{enumerate} \item Initialize $b$ to the zero vector and set $c^* = 1$.
\item {\bf While} $c^*>0$,
\begin{enumerate} 

\item {\bf For} $j=1, \ldots, n$, 
\begin{enumerate} \item If $b_j \neq 0$, set $c(j,p) = 0$ for all $p=1, \ldots, 2n+1$. 
\item  If $b_j = 0$, then for $p=1, \ldots, 2n+1$, set $\widetilde{b}_{j,p} = b + p {\bf e}_j$ and set $c(j,p) = {\rm rank}(C(A, \widetilde b_{j,p})) - {\rm rank}(C(A,b))$. 
\end{enumerate}
$~${\bf end}
\item Let $(j^*,p^*) \in \arg \max_{ (j,p) } c(j,p)$ and let $c^* = c(j^*,p^*)$. 
\item If $c^*>0$, set $b \leftarrow b + p^* {\bf e_{j^*}}$.
 \end{enumerate} 
$~${\bf end}
\item Output $b$. 
\end{enumerate}
\end{algorithm}

The deterministic algorithm has a simple interpretation. Rather than putting a random entry in the $j$'th place and 
testing the corresponding increase in the rank of the controllability matrix, it instead tries $2n+1$ different numbers to 
put in the $j$'th place, and chooses the one among them with the largest corresponding increase in rank. The main idea is that
to obtain the ``generic''' rank increase associated with a certain index we can either put a random number in the $j$'th
entry or we can simply test what happens when we try enough distinct values in the $j$'th entry.
 
We are  now finally ready to begin the proof of Theorem \ref{th:alg}.

\begin{proof}[Proof of Theorem \ref{th:alg}, first part] We will argue that  when the deterministic and the randomized algorithms enter the {\bf while} loop with identical $A,b$, we will
have that, for each $j \in \{1, \ldots, n\}$, $c(j)$ in the randomized algorithms equals $\max_p c(j,p)$ in the deterministic algorithm
with probability $1$.  Since this implies that Eq. (\ref{fdecrease}) holds for each run of the {\bf while} loop in the 
deterministic algorithm, we can simply repeat the proof of Lemma \ref{rand-lemma} to obtain that the deterministic algorithm obtains an $O( \log n)$ approximation. That the algorithm is polynomial follows since it terminates after at most $n$ steps of the {\bf while} loop, while each step requires at most $O(n^2)$ rank calculations, each of which may be done in polynomial time with Gaussian elimination. 

It remains to justify the claim about the equality of $c(j)$ in the randomized algorithm and $\max_p c(j,p)$ in the 
deterministic algorithm for each $j$. Fix $A,b,j$, and consider $p_m(t)$ which we define to be the sum of the squares  of all the $m \times m$  minors of the matrix $C(A,b+t {\bf e}_j)$. Naturally, $p_m(t)$ is a nonnegative polynomial in $t$ of degree at most $2n$.  Moreover, the statement $p_m(t) > 0$ is equivalent to the assertion that $C(A, b + t {\bf e}_j)$ has rank at least $m$. Note that because $p_m(t)$ is a polynomial in $t$ and therefore
has finitely many roots if not identically zero, we have that as long as $p_m(t)$ is not
identically zero, $p_m(t)$ evaluated a standard Gaussian is positive with probability $1$. Consequently, whenever the 
randomized algorithm enters step 2.a.ii and computes $c(j)$, we have that with probability $1$ this $c(j)$ equals $-{\rm rank}~ C(A,b)$
plus the largest $m \in \{1, \ldots, n\}$ such that $p_m(t)$ is not identically zero.

Similarly, consider the deterministic algorithm as it enters step 2.a. Observe that as long as $p_m(t)$ is not identically zero, plugging in one of $1,2,\ldots,2n+1$ into $p_m(t)$ will produce a positive number because 
$p_m(t)$ has degree at most $2n$. As a consequence, $\max_k c(j,k)$ equals $-{\rm rank}~ C(A,b)$ plus the largest $m$
such that $p_m(t)$ is not identically zero.
\end{proof}

\subsection{The minimal diagonal controllability problem} We now consider the related question of finding the sparsest diagonal matrix in 
${\cal B}'$. Our approach will be identical to the one we took in the case of the vector controllability problem: we will be adding nonzero 
elements to the diagonal matrix $B$ in a greedy way that increases the rank of the controllability matrix. We will prove that, in this case as well, this
yields an $O( \log n)$ approximation to the sparsest matrix in ${\cal B}'$. 

Since much of this section is a nearly identical repetition of the previous section, we will be correspondingly briefer. Our algorithm is stated in the following box;
we will refer to it as the {\em diagonal algorithm. }

\begin{algorithm} \begin{enumerate} \item Initialize $B$ to the zero matrix and set $c^* = 1$.
\item {\bf While} $c^*>0$,
\begin{enumerate} 

\item {\bf For} $j=1, \ldots, n$, 
\begin{enumerate} \item If $B_{jj} \neq 0$, set $c(j)=0$. \item  If $B_{jj} = 0$, set $\widetilde{B} = B + {\bf e}_j {\bf e}_j^T$, and  $c(j) = {\rm rank}(C(A, \widetilde B)) - {\rm rank}(C(A,B))$. 
\end{enumerate}
$~${\bf end}
\item Let $j^* \in \arg \max_j c(j)$ and let $c^* = c(j^*)$. 
\item If $c^*>0$, set $B \leftarrow B + {\bf e_{j^*}} {\bf e_{j^*}}^T$.
 \end{enumerate} 
$~${\bf end}
\item Output $B$. 
\end{enumerate}
\end{algorithm}

We will shortly conclude the proof of Theorem \ref{th:alg} by showing that the diagonal algorithms finds the approximation to the sparsest element in
${\cal B}'$ guaranteed by that theorem. In the course of the proof we will rely on the following elementary lemma. 

\begin{lemma} \label{drop} Suppose $S,X,Y$ are subspaces. If \begin{equation} \label{in1}  {\rm dim}~ (S \cap X \cap Y) \leq {\rm dim}~ (S \cap X) - a \end{equation} then 
\begin{equation} \label{in2}  {\rm dim}~ (S \cap Y) \leq {\rm dim}~ (S) - a. \end{equation}
\end{lemma} \begin{proof}  It is easy to see that if $A,B$ are subspaces, then 
\begin{equation} \label{intdim}   {\rm dim}~(A)  -  {\rm dim}~ (A \cap B) = {\rm dim}~ (A+B) -  {\rm dim}~ (B)  
\end{equation}  This is Theorem 2.18 in \cite{axler}. 

Now  from Eq. (\ref{intdim}), we have that Eq. (\ref{in1}) is true if and only if 
\[ a \leq  {\rm dim}~ (S \cap X + Y) - {\rm dim}~ Y  \] while again by Eq. (\ref{intdim}), we have that Eq. (\ref{in2}) is true if and only if 
\[ a \leq  {\rm dim}~ (S + Y ) - {\rm dim}~ Y   \] The lemma now follows from the  fact that  $S \cap X + Y \subset S + Y$. 
\end{proof}

\begin{proof}[Proof of Theorem \ref{th:alg}, second part] We will now prove that the diagonal algorithm above produces an $O(\log n)$ approximation to the sparsest matrix in ${\cal B}'$. Let us adopt the notation that $B' \in {\cal B}'$ is one such sparsest matrix; then the smallest number of nonzero
elements needed is $|B'|$. 

 We first remark that, by elementary considerations, given any diagonal matrix $D \in \R^{n \times n}$, we have that \begin{equation} \label{rankchar} {\rm rank}(C(A,D)) = \dim {\rm span} \{e^{At} e_{j} ~|~ D_{jj} \neq 0, t \in \R \},\end{equation} where the notion of a  span of a (possibly infinite) set is the usual one: it is  the collection of all finite linear combinations of its elements. 

Now let ${\cal J}$ be the set of all indices $j \in \{1, \ldots, n\}$ with $B_{jj}' \neq 0$. Since $B' \in {\cal B}'$ by assumption, we have that ${\rm span} ~ \{ e^{At} e_{j} ~|~ j \in {\cal J}, t \in \R \} = \R^n$. Obviously, we may restate this as follows: for any subspace $S$ of $\R^n$, we have that 
\[ S \cap {\rm span}~ \{ e^{At} e_j ~|~ j \in {\cal J}, t \in \R\}^{\perp} = \emptyset \] Since for any two sets $A,B$
\begin{equation} \label{spandim} {\rm span}~ (A \cup B)^{\perp} = {\rm span}~ (A)^\perp \cap {\rm span} (B)^{\perp} 
\end{equation} 
we have that 
\[ S \cap \bigcap_{j \in {\cal J}} {\rm span}~ \{ e^{At} e_j ~|~  t \in \R\}^{\perp}  = \emptyset \] Now consider what happens to the dimension of the subspace on the left-hand side of this equation as $j$ successively runs over all $|B'|$ elements of ${\cal J}$. Since the dimension drops from ${\rm dim}~(S)$ to $0$, it follows there is some index $j$ when the dimension drops by at least ${\rm dim}(S)/|B'|$. Consequently, there exists an index $j$ such that, denoting by ${\cal J}/j$ to be all the indices in ${\cal J}$ strictly less than $j$,  we have that

\begin{small}
\[  {\rm dim} \left( S \cap {\rm span} \{ e^{At} e_i ~|~ i \in {\cal J}/j, t \in \R\}^\perp  \cap {\rm span} \{ e^{At} e_j ~|~ t \in \R\}^\perp \right) \] \[ \leq 
{\rm dim} \left( S \cap {\rm span} \{ e^{At} e_i ~|~ i \in {\cal J}/j, t \in \R\}^\perp \right) - \frac{{\rm dim}~S}{|B'|}  \]
\end{small}

By Lemma \ref{drop}, this implies
\begin{equation} \label{existindex} {\rm dim}~ \left( S \cap  {\rm span}\{  e^{At} e_j ~|~ t \in \R\}^{\perp} \right) \leq {\rm dim}~(S) - \frac{ {\rm dim}~ S}{|B'|} 
\end{equation}

With this fact in place, consider the diagonal algorithm; specifically, consider each time the algorithm begins the ``while'' loop. Let us fix one such time and let $B$ be the matrix constructed by the algorithm up to date. Define the subspace $S(B)$ as $S(B) = {\rm span}~ \{e^{At} e_j ~|~ B_{jj} \neq 0, t \in \R\}^{\perp}$. Clearly, \[ {\rm dim}~ S(B) = n - {\rm rank}~ C(A,B) \] 

Now by Eq. (\ref{existindex}) if $S(B)$ is non-empty then there is one index $j$ with the property that \begin{small}
\[ {\rm dim} \left( S(B) \cap {\rm span}\{  e^{At} e_j ~|~ t \in \R\}^{\perp} \right) \leq {\rm dim}~(S(B)) - \frac{ {\rm dim}~ S(B)}{|B'|} \] \end{small} Thus appealing to Eq. (\ref{spandim}), we obtain setting this $B_{jj}$ to any nonzero number will decrease ${\rm dim}~ S(B)$ by a multiplicative factor of $1-1/|B'|$. It follows that the diagonal algorithm above decreases $S(B)$ by at least this much, since it picks an entry to set nonzero which maximizes the rank increase of the controllability matrix. 

Thus we have that after $O(|B|')$ iterations, ${\rm dim}~ S(B)$ decreases by a constant factor, and after $O(|B'| \log n)$ iterations we must have ${\rm dim}(S(B)) < 1$ and consequently $S(B) = \emptyset$.  Appealing finally
to Eq. (\ref{rankchar}), we get that a diagonal matrix $B$ is found rendering the system controllable whose number of nonzero entries is $O(|B'| \log n)$. 
\end{proof}

\section{Experiments with Erdos-Renyi random graphs\label{simul}} We report here on some simulations in the setting of Erdos-Renyi random graphs. These
graphs have have been the subject of considerable focus within the network controllability literature (see \cite{pu, LSB11}) due to their
use as models of biological interconnection networks \cite{er_biology}

We  simulate the randomized algorithm for finding the sparsest vector in ${\cal B}$. We used the MATLAB ``randn'' 
command to approximately produce a Gaussian random variable. To obtain the matrices $A$, we generated 
a directed Erdos-Renyi random graph where each (directed) link was present with a probability of $2 \log n/n$. We constructed the adjacency
matrices of these graphs and threw out those which had two eigenvalues which were at most $0.01$ apart. The purpose of this was to rule 
out matrices which had eigenspaces that were more than one dimensional, and due to the imprecise nature of eigenvalue computations with MATLAB, 
testing whether two eigenvalues were exactly equal did not ensure this. As a result of this process, the matrix  
was controllable with some vector $b$.

We then ran the randomized algorithm of the previous section to find a sparse vector making the system controllable. We note that
a few minor revisions to the algorithm were necessary due to the ill-conditioning of the controllability matrices. In particular, with
$A$ being the adjacency matrix of a directed Erdos-Renyi graph on as few as $20$ vertices and $b$ a randomly generated vector, the MATLAB command ``rank(ctrb(A,b))'' often returns nonsensical results due to the large condition number of the controllability matrix. We 
thus computed the rank of the controllability matrix by initially diagonalizing $A$ and counting the number of eigenvectors orthogonal
to $b$. 

We ran our simulation up to number of nodes $n=100$, generating $100$ random graphs for each $n$ and we found that almost all the matrices could be controlled with a $b$ with just one nonzero
entry, and all could be controlled with a $b$ that had two nonzero entries. Specifically, out of the $10,000$ random adjacency matrices generated in this way, 
$9,990$ could be controlled with a $1$-sparse $v$ and the remaining $10$ could be controlled with a $2$-sparse $b$. 

We conclude that our randomized protocol can often successfully find extremely sparse vectors making the system controllable. 
We are unable at the present moment to give a theoretical justification of the excellent performance observed here. Thus the investigation of our randomized
heuristic in various ``generic'' settings (such as the one simulated here)  is  currently an open problem.

\section{Conclusions\label{sec:concl}} We have shown that it is NP-hard to approximate the minimum number of variables in the system
that need to be affected for controllability within  a factor of $c \log n$ for some $c>0$, and we have provided an algorithm which approximates it to a factor of $c' \log n$
for some $c' > 0$. Up to the difference in the constants between $c$ and $c'$, this resolves the polynomial-time
approximability of the minimal controllability problem. 

The study of minimal controllability problems is relatively recent and quite a few open questions remain. 
For example, it is reasonable
to anticipate that real-world networks may have certain ``generic'' features which considerably simplify the minimal controllability
problem. It would therefore be interesting to find algorithms for minimal controllability which always return the correct answers and
which ``generically'' finish in polynomial time. Finally, an understanding of more nuanced features of controllability (i.e., how the energy required to move a network
from one state to another depends on its structure) appears to be needed (see \cite{Motter1, bullo, add1, add2, add3} for work on this problem).

\begin{small}

\end{small}
\end{document}